\documentclass[12pt]{article}
\usepackage[T1]{fontenc}
\usepackage{lmodern,amsmath,amssymb,amsthm,mathtools}
\usepackage[english]{babel}
\usepackage[margin=1in]{geometry}
\usepackage{booktabs,array,microtype}
\usepackage[colorlinks=true,allcolors=blue,breaklinks=true]{hyperref}
\hypersetup{pdftitle={Ergodicity, memory loss, and directional stability on von Neumann algebras},pdfauthor={Perplexity Computer}}
\newtheorem{theorem}{Theorem}[section]
\newtheorem{proposition}[theorem]{Proposition}
\newtheorem{corollary}[theorem]{Corollary}
\newtheorem{lemma}[theorem]{Lemma}
\theoremstyle{definition}
\newtheorem{definition}[theorem]{Definition}
\newtheorem{example}[theorem]{Example}
\newtheorem{remark}[theorem]{Remark}
\newcommand{\M}{\mathcal M}
\newcommand{\Sn}{\mathcal S_{\mathcal N}(\M)}
\newcommand{\one}{\mathbf 1}
\newcommand{\id}{\mathrm{id}}
\newcommand{\R}{\mathbb R}
\newcommand{\N}{\mathbb N}
\newcommand{\norm}[1]{\left\lVert#1\right\rVert}
\DeclareMathOperator{\osc}{osc}
\DeclareMathOperator{\Fix}{Fix}
\DeclareMathOperator{\Tr}{Tr}
\DeclareMathOperator{\diam}{diam}
\DeclareMathOperator{\ran}{ran}

\allowdisplaybreaks
\title{Ergodicity, memory loss, and directional stability on von Neumann algebras}

\author{%
  Abdessatar Souissi\thanks{Email: \texttt{a.souaissi@qu.edu.sa}}\\
  \small Department of Mathematical Physics, Preparatory Institute for Engineering studies,\\
  \small University of Monastir, Monastir, Tunisia\\
  \and
  Abdessatar Barhoumi\thanks{Email: \texttt{abarhoumi@kfu.edu.sa}}\\
  \small Department of Mathematics and Statistics, College of Science,\\
  \small King Faisal University, Al-Ahsa P.O. Box 400, 31982, Saudi Arabia
}
\date{}

\begin{document}
\maketitle
\tableofcontents
\begin{abstract}
For a deterministic sequence of normal unital completely positive maps on a von Neumann algebra, convergence of normal states, loss of memory, and convergence of Ces\`aro averages are different properties. This paper separates these properties in the predual norm and identifies the precise role of composition order. Uniform loss of memory for right-appended products on the predual implies uniform convergence to a normal state without any compactness or trace assumption; the convergence rate is bounded by the image diameter. In contrast, on every infinite-dimensional $\sigma$-finite von Neumann algebra there are right-appended products with pointwise loss of memory but no convergent normal-state trajectory, even though all nontrivial state images equal a fixed compact line segment. Consequently, finite dimensionality is characterized by the universal implication from pointwise backward memory loss to mixing. An attained normal-functional minorization coefficient extends the common-lower-bound construction for matrix channels to arbitrary von Neumann algebras and yields quantitative contraction and unique normal entrance laws. Further results distinguish pointwise from uniform relaxation, Ces\`aro convergence from mixing, fixed-start from time-uniform contraction, and stationary from moving asymptotic states. As an application, inhomogeneous transfer states on spin chains, including matrix product states, are shown to have a unique thermodynamic limit, uniformly in the boundary condition, with an arbitrary von Neumann algebra as bond algebra; boundary dependence and correlations are controlled by the same product coefficients. All counterexamples are explicit and extend to algebras of types I, II, and III.
\end{abstract}

\section{Introduction}\label{sec:intro}


The ergodic theory of quantum dynamics has developed along two largely separate lines. The first is operator algebraic. Here the dynamics is a normal positive map, an automorphism, or a semigroup on a von Neumann algebra, usually with a distinguished invariant normal state. Mean and individual ergodic theorems were obtained by Lance for automorphisms and convex sets of states \cite{Lan76}, and by Yeadon for semifinite algebras in the noncommutative $L^1$ setting \cite{Yea80}. Junge and Xu later proved maximal ergodic inequalities in noncommutative $L^p$-spaces \cite{JX07}. The long-time behaviour of dynamical semigroups on $W^*$-algebras was analysed by Frigerio and Verri \cite{FV82}. Kümmerer's theory of Markov dilations embeds completely positive dynamics with a faithful invariant state into reversible dynamics \cite{Kum85}. Fidaleo studied strong ergodic properties, including unique ergodicity and uniform convergence of Ces\`aro averages \cite{Fid09}. More recent work on the structure of ergodic dynamics includes the existence of ergodic states on type $\mathrm{III}_1$ factors, due to Marrakchi and Vaes \cite{MV24}, and a multiplicative ergodic theorem for cocycles with values in a finite von Neumann algebra, due to Bowen, Hayes and Lin \cite{BHL21}. In this tradition the dynamics is typically homogeneous, and convergence is studied in weak-star, strong, or noncommutative $L^p$ topologies relative to a fixed state.

The second line comes from quantum information theory. A quantum channel on a matrix algebra is ergodic or mixing according to spectral properties of its invariant states and its peripheral spectrum \cite{BCGPY}. The peripheral structure, period, and cyclic decomposition of a channel were analysed by Carbone and Jen\v{c}ov\'a \cite{CJ20}. Albert related the asymptotic projection of a channel to its conserved quantities and to matrix product states \cite{Alb19}. Quantitative contraction is central in this line. Strict contractivity of noise channels limits fault-tolerant computation \cite{Rag02}. Quantum Doeblin coefficients give computable upper bounds on trace-distance contraction coefficients, with applications to time-varying channels \cite{Hir24,GHNW25}. Within quantum probability, Accardi, Lu and Souissi proved a Markov--Dobrushin inequality for channels \cite{ALS21}. Souissi and Barhoumi showed that a positive Markov--Dobrushin constant implies exponential mixing, and applied this to matrix product states \cite{SB25}.


When the channel changes in time, the iterates of a single map are replaced by ordered products, and no invariant state need exist. In classical probability this situation is governed by Dobrushin's ergodicity coefficient. Its extension to abstract state spaces and to von Neumann algebra preduals is due to Mukhamedov \cite{Muk13,Muk15}. Mukhamedov and Al-Rawashdeh treated weak and uniform ergodicity of nonhomogeneous chains relative to a projection \cite{MA22}. Gaubert and Qu established the duality between contraction and oscillation for Markov operators on cones \cite{GQ}.

A second group of results concerns random products. Bruneau, Joye and Merkli studied random repeated interaction systems \cite{BJM08}. Movassagh and Schenker obtained an ergodic theorem for processes driven by an ergodic dynamical system, with exponential convergence to an ergodic sequence of states that does not depend on the initial state \cite{MS21,MS22}. Ekblad and Schenker developed a reducibility theory without the decoherence hypothesis of this framework \cite{ES}. Nelson and Roon extended the stationary random theory to finite von Neumann algebras, using a projective metric on the cone of positive elements \cite{NR}. Pathirana and Schenker studied continuous-time inhomogeneous Markovian dynamics in a stationary random environment \cite{PS25}. Ergodic theorems for quantum trajectories form a related direction \cite{KM04,BFPP19,EMP25}.

For deterministic inhomogeneous channels on matrix algebras, Souissi \cite{Sou26} separated the forward and backward composition orders. That paper showed that weak and strong mixing coincide for backward products, which are automatically nested, and derived convergence rates from common positive lower bounds of channel outputs. Pathirana \cite{P26} developed a product-level trace-Dobrushin theory for finite-dimensional channel products. In the deterministic case, decay of the product coefficient is equivalent to asymptotic replacement by a moving rank-one channel; for two-sided products, pullback forgetting gives a unique boundary state.


Quantum Markov chains originate in Accardi's noncommutative Markov property \cite{Acc75}. Later constructions include entangled Markov chains \cite{AF05} and a unifying framework for the existing definitions \cite{ASS20}. Finitely correlated states were introduced by Fannes, Nachtergaele and Werner \cite{FNW92}, and their pure case was analysed in \cite{FNW94}. Matrix product states are their pure vector counterparts, and their representation theory was systematised in \cite{PVWC07}.

In all these constructions the state of a spin chain is produced by iterating a transfer map on an auxiliary, or bond, algebra. Consequently, the thermodynamic limit, boundary independence, and decay of correlations are statements about products of channels. For translation-invariant or ergodic chains this connection appears in \cite{MS22,SB25}. For non-translation-invariant chains, \cite{Sou26} expresses the infinite-volume state through a backward sequence of boundary states, and \cite{P26} obtains infinite-volume limits and correlation bounds from product coefficients. Each of these works assumes a finite-dimensional bond space.

The transfer construction itself does not require this. It defines states for an arbitrary auxiliary von Neumann algebra, for instance $B(K)$ with $K$ infinite dimensional, or a $\mathrm{II}_1$ factor. Finite dimensionality is what makes such a state ``finitely'' correlated in the sense of \cite{FNW92}, but it is not needed for the construction. The theory developed below applies directly to this setting.


This paper asks which parts of the finite-dimensional, deterministic theory of \cite{Sou26} remain valid on an arbitrary von Neumann algebra $\M$, and which fail. Three features of the infinite-dimensional setting are decisive. The normal state space is in general not norm compact. There may be no trace, no density operators, and no normal pure states. Finally, in applications to spin chains, the bond algebra itself can be infinite dimensional.

The predual Dobrushin coefficient and its submultiplicativity are known \cite{Muk13,Muk15} and are not presented as new. The contribution of this paper is to determine exactly where directional stability holds for arbitrary deterministic sequences of normal channels. All composition directions refer to maps on states, not to maps on observables. Throughout, $\Sn$ denotes the set of \emph{normal} states of $\M$; the subscript $\mathrm n$ stands for ``normal'' and is not an integer index (see Section~\ref{sec:setting}).

The first two main results describe this boundary.

\begin{theorem}[Uniform backward stability]\label{thm:intro-uniform}
Let $\M$ be a von Neumann algebra, let $T_n=(\Phi_n)_*$ be the preduals of normal unital completely positive maps, and put $B_n=T_0\cdots T_{n-1}$. Then
\[
 \sup_{\varphi,\psi\in\Sn}\norm{B_n\varphi-B_n\psi}\longrightarrow0
\]
if and only if there exists $\omega\in\Sn$ such that
\[
 \sup_{\varphi\in\Sn}\norm{B_n\varphi-\omega}\longrightarrow0.
\]
In that case, for every $n$,
\[
 \sup_{\varphi\in\Sn}\norm{B_n\varphi-\omega}
 \leq \sup_{\varphi,\psi\in\Sn}\norm{B_n\varphi-B_n\psi}.
\]
No norm compactness of the state space or of a channel image is required.
\end{theorem}

\begin{theorem}[A finite-dimensionality characterization]\label{thm:intro-dimension}
For a nonzero $\sigma$-finite von Neumann algebra $\M$, the following are equivalent:
\begin{enumerate}
\item $\M$ is finite dimensional.
\item For every sequence of normal unital completely positive maps, pointwise loss of memory of the backward state products $B_n$ implies mixing to a normal state.
\item The implication in (2) holds for every such sequence for which $\bigcup_{n\geq1}B_n(\Sn)$ is norm relatively compact.
\end{enumerate}
If $\M$ is infinite dimensional, the counterexample can be chosen so that $B_n(\Sn)$ is the same compact line segment for every $n\geq1$, while the process has pointwise loss of memory and is not even Ces\`aro ergodic.
\end{theorem}

Theorem~\ref{thm:intro-uniform} is a dimension-free positive result. Theorem~\ref{thm:intro-dimension} shows that uniformity cannot be replaced by pointwise convergence: the matrix theorem of \cite{Sou26} holds pointwise only because the state space of a matrix algebra is compact.

The third main result applies the backward theory to spin chains with an infinite-dimensional bond algebra. Let $\mathcal B$ be a von Neumann algebra, let $d\geq1$, and let
\[
 \mathcal E_k:M_d(\mathbb C)\otimes\mathcal B\longrightarrow\mathcal B,\qquad k\geq1,
\]
be normal unital completely positive maps. Put $\Phi_k(b)=\mathcal E_k(\one\otimes b)$ and $T_k=(\Phi_k)_*$. The maps $\mathcal E_k$ define, for every normal boundary state $\sigma$ at site $N$, a finite-volume state $\varphi_N^\sigma$ on $\mathcal A_{[1,N]}=M_d(\mathbb C)^{\otimes N}$; see Section~\ref{sec:FCS}.

\begin{theorem}[Inhomogeneous transfer states]\label{thm:intro-FCS}
Assume that $\delta(T_m\cdots T_{n-1})\to0$ as $n\to\infty$ for every $m\geq1$. Then there is a unique state $\varphi_\infty$ on the quasi-local algebra $\mathcal A=\overline{\bigcup_N\mathcal A_{[1,N]}}^{\norm\cdot}$ such that, for all $m<N$ and all normal states $\sigma$ of $\mathcal B$,
\[
 \bigl\|\varphi_N^\sigma|_{\mathcal A_{[1,m]}}-\varphi_\infty|_{\mathcal A_{[1,m]}}\bigr\|
 \leq2\,\delta(T_{m+1}\cdots T_N).
\]
Its local restrictions are expressed through the unique normal entrance law of the backward products. For local observables $x$ on $[1,m]$ and $z$ on $[m+\ell+1,m+\ell+r]$,
\[
 \bigl|\varphi_\infty(x\otimes\one_\ell\otimes z)-\varphi_\infty(x)\varphi_\infty(z)\bigr|
 \leq4\,\delta(T_{m+1}\cdots T_{m+\ell})\,\norm x\norm z.
\]
\end{theorem}

This theorem extends the boundary-state representation of \cite{Sou26} and the finite-dimensional replacement estimates of \cite{P26} to an arbitrary bond algebra. The estimate is uniform in the boundary state, and it cannot be obtained from pointwise memory loss (Remark~\ref{rem:FCS-pointwise}).

Further results identify the precise role of each hypothesis. An attained normal minorization coefficient replaces the rank-one construction of \cite{Sou26} and controls $\delta$ in any von Neumann algebra (Section~\ref{sec:minorization}). Unique normal entrance laws and forward tracking estimates are proved in Sections~\ref{sec:direction} and~\ref{sec:rates}. Explicit counterexamples separate pointwise from uniform convergence, Ces\`aro convergence from mixing, convergence from a fixed starting time from uniformity over all starting times, and a single strict contraction of a prefix from recurrent contraction (Sections~\ref{sec:dimension} and~\ref{sec:homogeneous}). A Hilbert-space criterion for processes that preserve a faithful normal state appears in Section~\ref{sec:L2}.


Section~\ref{sec:setting} fixes notation and composition orders, and Section~\ref{sec:hierarchy} states the hierarchy of convergence notions. Sections~\ref{sec:delta} and~\ref{sec:minorization} treat the contraction and minorization coefficients. Theorem~\ref{thm:intro-uniform} is proved in Section~\ref{sec:direction}, and quantitative estimates follow in Section~\ref{sec:rates}. Theorem~\ref{thm:intro-dimension} and the universal counterexamples are in Section~\ref{sec:dimension}, with homogeneous comparisons in Section~\ref{sec:homogeneous}. Section~\ref{sec:L2} gives the Hilbert-space criterion, Section~\ref{sec:FCS} proves Theorem~\ref{thm:intro-FCS}, and Section~\ref{sec:conclusion} summarises the conclusions.

\section{Preduals, channels, and composition order}\label{sec:setting}

Let $\M$ be a nonzero von Neumann algebra, with unit $\one$, predual $\M_*$, self-adjoint part $\M_{*,\mathrm{sa}}$, and positive cone $\M_{*,+}$. The \emph{normal state space} of $\M$ is
\[
 \Sn=\{\varphi\in\M_{*,+}:\varphi(\one)=1\},
\]
that is, the set of states of $\M$ that are normal (equivalently, $\sigma$-weakly continuous, or completely additive on orthogonal families of projections). The subscript $\mathrm n$ abbreviates ``normal''. When $\M=B(H)$, normal states are exactly the functionals $x\mapsto\Tr(\rho x)$ with $\rho$ a density operator, so $\Sn$ is the von Neumann algebra counterpart of the set $\mathcal S(H)$ of density operators. Also write
\[
 E_0=\{f\in\M_{*,\mathrm{sa}}:f(\one)=0\}.
\]
All norms on functionals are predual norms. In particular, $\norm{\varphi}=\varphi(\one)$ for $\varphi\geq0$. When $\M$ is $\sigma$-finite, a faithful normal state may be denoted by $\chi$; no trace is assumed. The general positive results below do not require $\sigma$-finiteness. That assumption is convenient for the universal counterexamples.

A channel on observables is a normal unital completely positive map $\Phi:\M\to\M$. Its predual $T=\Phi_*$ is defined by $(Tf)(x)=f(\Phi(x))$. It is a contraction and preserves $\Sn$. In a general von Neumann algebra the intrinsic normalization statement is
\[
 (Tf)(\one)=f(\one),
\]
not preservation of an unspecified trace. Density operators and a trace become available only after an appropriate additional representation or semifinite hypothesis.

Fix a sequence $(\Phi_n)_{n\geq0}$ and put $T_n=(\Phi_n)_*$. For $m<n$ define
\begin{equation}\label{eq:FB}
 F_{m,n}=T_{n-1}\cdots T_m,\qquad
 B_{m,n}=T_m\cdots T_{n-1},\qquad F_{m,m}=B_{m,m}=\id.
\end{equation}
Thus $F_{m,n}$ is the chronological state evolution, and $B_{m,n}$ is the reverse-ordered state product. The terms forward and backward will always mean these two state-space orders. Their observable duals are, respectively,
\begin{equation}\label{eq:duality-orders}
 F_{m,n}^*=\Phi_m\cdots\Phi_{n-1},\qquad
 B_{m,n}^*=\Phi_{n-1}\cdots\Phi_m.
\end{equation}
Here the superscript $*$ denotes the Banach adjoint, with $\M=(\M_*)^*$.

The reversal is essential: $(\Phi\circ\Psi)_*=\Psi_*\circ\Phi_*$. In particular, the observable product $\Phi_{n-1}\cdots\Phi_0$, which is chronological in the Heisenberg picture, has the backward state product $B_{0,n}$ as its predual. Formula~\eqref{eq:duality-orders} removes any ambiguity in the terms forward and backward. Neither product is an inverse dynamics, a Petz reversal, or a modular adjoint. Such notions would require additional data and are not used here.

The composition identities are
\[
 F_{m,n}=F_{k,n}F_{m,k},\qquad
 B_{m,n}=B_{m,k}B_{k,n}\quad(m\leq k\leq n).
\]
For backward products the state images satisfy
\begin{equation}\label{eq:nesting}
 B_{m,n+1}(\Sn)\subseteq B_{m,n}(\Sn).
\end{equation}
For forward products one instead has
$F_{m,n+1}=T_nF_{m,n}$, which need not give nested images. This is an order-of-composition distinction; it occurs already for classical noncommuting stochastic matrices and is not exclusively quantum.

\section{Convergence notions and the correct hierarchy}\label{sec:hierarchy}

Fix $m$ and one of the families $P_n=F_{m,n}$ or $P_n=B_{m,n}$, and write
\[
 A_N=\frac1N\sum_{j=0}^{N-1}P_{m+j},\qquad P_m=\id.
\]
The initial index will be suppressed when $m=0$.

\begin{definition}\label{def:hierarchy}
The family $(P_n)$ is \emph{mixing} if there exists $\omega\in\Sn$ with $\norm{P_n\varphi-\omega}\to0$ for every $\varphi\in\Sn$. It is \emph{uniformly mixing} if this convergence is uniform over $\Sn$, and \emph{exponentially mixing} if, for constants $C<\infty$ and $0<r<1$,
\[
 \sup_{\varphi\in\Sn}\norm{P_n\varphi-\omega}\leq Cr^{n-m}.
\]
It is \emph{mean ergodic to one state} if $\norm{A_N\varphi-\omega}\to0$ for all normal states, and \emph{uniformly mean ergodic to one state} if the convergence is uniform.

It has \emph{pointwise loss of memory} if
\[
 \norm{P_n\varphi-P_n\psi}\longrightarrow0\quad(\varphi,\psi\in\Sn),
\]
and \emph{uniform loss of memory} if
$d_n:=\sup_{\varphi,\psi\in\Sn}\norm{P_n\varphi-P_n\psi}\to0$.
The exponential version is $d_n\leq Cr^{n-m}$. Mean loss of memory and uniform mean loss of memory are defined by replacing $P_n$ by $A_N$ in these two conditions.
\end{definition}

The phrases ``mean ergodic'' and ``ergodic'' below refer to convergence to a single state, not merely existence of an arbitrary mean-ergodic projection. The term ``weak mixing'' in \cite{Sou26} corresponds to the norm loss-of-memory property just defined. It does not mean convergence in a weak topology, and is not the usual ergodic-theoretic weak mixing defined by Ces\`aro decay of absolute correlations. Keeping these notions separate prevents an otherwise substantial terminological ambiguity.

\begin{proposition}[Hierarchy]\label{prop:hierarchy}
For either composition order,
\[
 \text{exponential mixing}\Longrightarrow\text{uniform mixing}
 \Longrightarrow\text{mixing}\Longrightarrow\text{mean ergodicity},
\]
and
\[
 \text{uniform mixing}\Longrightarrow\text{uniform mean ergodicity}
 \Longrightarrow\text{mean ergodicity}.
\]
The analogous implications hold for loss of memory. Each convergence-to-one-state property implies its corresponding loss-of-memory property. In finite dimension, pointwise and uniform versions coincide, both for products and for averages. In infinite dimension, mixing need not imply uniform mean ergodicity.
\end{proposition}
\begin{proof}
Convergence to a common state gives loss of memory by the triangle inequality. Ces\`aro summation preserves pointwise norm convergence, and also preserves convergence of the relevant suprema when these converge to zero. To prove the finite-dimensional assertion, use that all channels and their averages are contractions. If $P_n\varphi\to\omega$ pointwise but not uniformly, choose $\varphi_n$ violating uniform convergence and pass to a norm-convergent subsequence. Contractivity gives
\[
 \norm{P_n\varphi_n-\omega}
 \leq\norm{\varphi_n-\varphi}+\norm{P_n\varphi-\omega}\longrightarrow0,
\]
a contradiction. Apply the same argument to pairs of states and to averages. The final assertion follows from Example~\ref{ex:countdown}.
\end{proof}

Thus the correct infinite-dimensional diagram branches: mixing implies mean ergodicity, but there is no general arrow from mixing to uniform mean ergodicity. In finite dimension the appropriate chain is
\begin{multline*}
 \text{exponential mixing}\Longrightarrow
 \text{uniform mixing}\Longleftrightarrow\text{mixing}\\
 \Longrightarrow\text{uniform mean ergodicity}\Longleftrightarrow
 \text{mean ergodicity}.
\end{multline*}
The two remaining reverse arrows fail for inhomogeneous processes; explicit examples appear below. For a homogeneous process, uniform mixing always has an exponential rate; see Theorem~\ref{thm:homogeneous}.

\section{The predual contraction coefficient}\label{sec:delta}

For a channel predual $T$ define
\begin{equation}\label{eq:delta}
 \delta(T)=\norm{T|_{E_0}}
 =\sup_{0\ne f\in E_0}\frac{\norm{Tf}}{\norm f},
\end{equation}
with $\delta(T)=0$ if $E_0=\{0\}$. This is the usual predual Dobrushin coefficient, as in \cite{Muk13,Muk15}. The following elementary facts are included to make all subsequent bounds self-contained.

\begin{lemma}\label{lem:delta}
For channel preduals $T,S$,
\begin{equation}\label{eq:diameter}
 0\leq\delta(T)\leq1,\qquad
 \delta(TS)\leq\delta(T)\delta(S),\qquad
 2\delta(T)=\sup_{\varphi,\psi\in\Sn}\norm{T\varphi-T\psi}.
\end{equation}
For $x=x^*\in\M$ put
$\osc(x)=\sup\operatorname{spec}(x)-\inf\operatorname{spec}(x)$.
Then
\begin{equation}\label{eq:oscillation}
 \delta(T)=\sup_{x=x^*,\,\osc(x)>0}
       \frac{\osc(T^*x)}{\osc(x)}.
\end{equation}
\end{lemma}
\begin{proof}
A self-adjoint normal functional $f$ has a normal Jordan decomposition $f=f_+-f_-$ with $\norm f=f_+(\one)+f_-(\one)$. If $f\in E_0$, then $f_+(\one)=f_-(\one)=\norm f/2$. For $f\ne0$, therefore,
\[
 f=\frac{\norm f}{2}(\varphi-\psi)
 \quad\text{for some }\varphi,\psi\in\Sn.
\]
This gives $\delta(T)\leq\frac12\sup\norm{T\varphi-T\psi}$. Conversely,
$\norm{T\varphi-T\psi}\leq\delta(T)\norm{\varphi-\psi}\leq2\delta(T)$.
The coefficient is at most one by contractivity; submultiplicativity follows because $S(E_0)\subseteq E_0$.

For the last assertion, the real Banach dual of $E_0$ is $\M_{\mathrm{sa}}/\R\one$, with quotient norm
\[
 \inf_{c\in\R}\norm{x-c\one}=\tfrac12\osc(x).
\]
Indeed the annihilator of $E_0$ is $\R\one$, and centering the spectrum gives the displayed equality. The adjoint of $T|_{E_0}$ is the induced map of this quotient. Equality of the norms of an operator and its adjoint proves~\eqref{eq:oscillation}.
\end{proof}

All operator norms in statements using $E_0$ are real operator norms. Uniform convergence over states is equivalent to operator norm convergence on $\M_{*,\mathrm{sa}}$ to the rank-one map
\[
 R_\omega f=f(\one)\omega.
\]
In fact the operator norm of $T-R_\omega$ on this real space equals its supremum on states, by Jordan decomposition. On the complex predual these two norms are equivalent with a factor at most two, which is sufficient for convergence but should not be silently replaced by an exact identity.

\section{An attained normal minorization coefficient}\label{sec:minorization}

\begin{definition}
For a channel predual $T$ let
\[
 \mathcal K(T)=\{\kappa\in\M_{*,+}:\kappa\leq T\varphi
                        \text{ for every }\varphi\in\Sn\},\qquad
 \beta(T)=\sup_{\kappa\in\mathcal K(T)}\kappa(\one).
\]
The order is the ordinary order of normal positive functionals. Thus $\kappa\leq\lambda$ means $\kappa(x)\leq\lambda(x)$ for all $x\in\M_+$.
\end{definition}

\begin{theorem}[Normal common-lower-bound theorem]\label{thm:minorization}
For every von Neumann algebra and every normal channel predual $T$, the maximum defining $\beta(T)$ is attained. Moreover,
\begin{equation}\label{eq:beta}
 0\leq\beta(T)\leq1,\qquad \delta(T)\leq1-\beta(T).
\end{equation}
A normal positive functional $\kappa$ belongs to $\mathcal K(T)$ if and only if
\begin{equation}\label{eq:heisenberg-minorization}
 T^*(x)\geq\kappa(x)\one\qquad(x\in\M_+).
\end{equation}
When $\M=B(H)$ with $\dim H<\infty$, writing $\kappa(x)=\Tr(kx)$ identifies $\beta(T)$ with the maximal trace of a common positive lower bound of the outputs of all rank-one input states. Thus it recovers the variational common-lower-bound coefficient of \cite{Sou26}, with $\eta_{\mathrm{MD}}=1-\beta(T)$.
\end{theorem}
\begin{proof}
The set $\mathcal K(T)$ contains zero. Fix $\varphi_0\in\Sn$ and put $\lambda=T\varphi_0$. Regard the order interval $[0,\lambda]$ as a subset of the full Banach dual $\M^*$. It is $\sigma(\M^*,\M)$ closed and contained in its closed unit ball, hence compact by Banach--Alaoglu. Every $f$ in this interval is normal: if $x_i\uparrow x$ in $\M_+$, then
\[
 0\leq f(x-x_i)\leq\lambda(x-x_i)\longrightarrow0.
\]
For every normal state $\varphi$, the constraint $f\leq T\varphi$ is weak-star closed, since it is the intersection over $x\in\M_+$ of the scalar closed inequalities $f(x)\leq(T\varphi)(x)$. Consequently $\mathcal K(T)$ is a compact subset of $[0,\lambda]$ in this topology. Evaluation at $\one$ is continuous and therefore has a maximum. Its value lies in $[0,1]$.

For admissible $\kappa$, let $b=\kappa(\one)$. The map
\[
 S=T-R_\kappa,\qquad R_\kappa(f)=f(\one)\kappa,
\]
is positive: for $f\geq0$, normalize $f$ to a state when it is nonzero. For such $f$, $(Sf)(\one)=(1-b)f(\one)$. A Jordan decomposition therefore gives $\norm{Sf}\leq(1-b)\norm f$ on $\M_{*,\mathrm{sa}}$. On $E_0$, $S=T$, so $\delta(T)\leq1-b$. Take a maximizing $\kappa$.

The order condition is equivalent to
$\varphi(T^*(x))\geq\kappa(x)$ for every normal state and every positive $x$. Normal states determine the order in a von Neumann algebra: a nonzero negative spectral projection of a self-adjoint element supports a normal state detecting the negative part. This proves~\eqref{eq:heisenberg-minorization}.

In finite dimension, normal states are convex combinations of rank-one states. A lower bound for all rank-one outputs is consequently a lower bound for all state outputs, and the converse is immediate. The density-functional correspondence preserves order and identifies $\kappa(\one)$ with $\Tr(k)$.
\end{proof}

\begin{remark}
The compactness used in the proof is the weak-star compactness of an order interval inside $\M^*$, not norm compactness of $\Sn$. Domination by one fixed normal functional prevents a singular functional from appearing in that interval. The maximizing functional need not be unique, but its maximal mass is a scalar and is unambiguous. No lattice infimum in the noncommutative order is asserted.
\end{remark}

\begin{example}[Strict contraction without a common lower bound]\label{ex:beta-not-delta}
On $\ell^\infty(\{1,2,3\})$, take a classical channel with output laws
\[
 p_1=(1/2,1/2,0),\qquad p_2=(0,1/2,1/2),\qquad
 p_3=(1/2,0,1/2).
\]
The smallest value of each output coordinate over the three inputs is zero, so $\beta(T)=0$. Pairwise distances are one, and convexity shows that the diameter of the output simplex is one. Lemma~\ref{lem:delta} gives $\delta(T)=1/2$. Thus $1-\beta(T)$ is a sufficient contraction bound, not the optimal coefficient in general.
\end{example}

\begin{example}[Positive subtraction is not CP subtraction]\label{ex:positive-not-cp}
On $M_2(\mathbb C)$, in the Schr\"odinger picture, let
\[
 T(\rho)=\frac{\Tr(\rho)I+\rho^{\mathsf T}}3.
\]
Its Choi matrix is $(I\otimes I+\mathsf F)/3\geq0$, where $\mathsf F$ is the flip; also $\Tr(T\rho)=\Tr\rho$. Every state output dominates $I/3$. Conversely, given a unit vector $v$, choose a pure input $\rho$ with $\rho^{\mathsf T}v=0$; any common lower bound $k$ then satisfies $\langle v,kv\rangle\leq1/3$. Hence the maximal lower bound is $I/3$ and $\beta(T)=2/3$. But
\[
 T(\rho)-\Tr(\rho)I/3=\rho^{\mathsf T}/3
\]
is not completely positive: its Choi matrix $\mathsf F/3$ has a negative eigenvalue. Thus Theorem~\ref{thm:minorization} only uses positivity of the residual. It must not be read as a CP-order Doeblin decomposition.
\end{example}

\section{Directional convergence and normal entrance laws}\label{sec:direction}

\begin{theorem}[Nested images with shrinking diameter]\label{thm:nested}
Let $P_n$ be channel preduals on a von Neumann algebra. Suppose that
$P_{n+1}(\Sn)\subseteq P_n(\Sn)$ and $d_n=\diam(P_n(\Sn))\to0$.
Then there exists a unique $\omega\in\Sn$ such that
\[
 \sup_{\varphi\in\Sn}\norm{P_n\varphi-\omega}\leq d_n.
\]
The sets $P_n(\Sn)$ need not be closed. The state $\omega$ belongs to the norm closure of each set $P_n(\Sn)$.
\end{theorem}
\begin{proof}
Choose $\varphi_0\in\Sn$ and set $z_n=P_n\varphi_0$. If $k\geq n$, nesting puts $z_k,z_n$ in $P_n(\Sn)$, so $\norm{z_k-z_n}\leq d_n$. Thus $(z_n)$ is norm Cauchy. The predual is complete, and its positive cone and normalization hyperplane are norm closed, hence its limit $\omega$ is a normal state. For any $\varphi\in\Sn$ and $k\geq n$, the two states $P_n\varphi,z_k$ lie in $P_n(\Sn)$, so $\norm{P_n\varphi-z_k}\leq d_n$. Pass to the limit in $k$. The bound proves both uniqueness and membership in each image closure.
\end{proof}

Apply this theorem to~\eqref{eq:nesting}, and use~\eqref{eq:diameter}. It proves Theorem~\ref{thm:intro-uniform}. In particular, for every fixed $m$, uniform backward mixing is equivalent to $\delta(B_{m,n})\to0$, and exponential uniform backward memory loss implies exponential backward mixing with the same diameter bound. This dimension-free conclusion does not extend to pointwise memory loss, as Section~\ref{sec:dimension} will show.

\begin{corollary}[Forward nesting]\label{cor:forward-nesting}
If $F_{0,n+1}(\Sn)\subseteq F_{0,n}(\Sn)$ for all $n$ and the forward process has uniform loss of memory, then it is uniformly mixing, with $\sup_\varphi\norm{F_{0,n}\varphi-\omega}\leq2\delta(F_{0,n})$. In particular exponential uniform loss of memory and nesting give exponential mixing with the same rate, in either direction.
\end{corollary}
This is the infinite-dimensional form of the nesting principle of \cite{Sou26}. In finite dimension pointwise memory loss suffices there, by Proposition~\ref{prop:hierarchy}. In infinite dimension the hypothesis must be uniform: the backward products of Theorem~\ref{thm:reservoir} below are nested and lose memory pointwise, yet do not converge.

\begin{theorem}[Entrance laws]\label{thm:entrance}
Suppose that, for every fixed $m\geq0$, $\delta(B_{m,n})\to0$ as $n\to\infty$. Then there is a unique sequence of normal states $(\omega_m)_{m\geq0}$ satisfying
\begin{equation}\label{eq:entrance}
 \omega_m=T_m\omega_{m+1}\qquad(m\geq0).
\end{equation}
For this sequence,
\begin{equation}\label{eq:entrance-estimate}
 \sup_{\varphi\in\Sn}\norm{B_{m,n}\varphi-\omega_m}
 \leq2\delta(B_{m,n}).
\end{equation}
More generally, the same estimate holds with an arbitrarily chosen boundary state $\varphi_n$ at time $n$ in place of a fixed $\varphi$.
\end{theorem}
\begin{proof}
For each $m$, Theorem~\ref{thm:nested} gives a limiting normal state $\omega_m$ and the bound. Since $B_{m,n}=T_mB_{m+1,n}$, norm continuity gives~\eqref{eq:entrance}. If $(\nu_m)$ is another normal entrance law, iteration gives $\nu_m=B_{m,n}\nu_n$, hence
$\norm{\nu_m-\omega_m}\leq2\delta(B_{m,n})\to0$. The uniform bound also proves the assertion about arbitrary boundary sequences.
\end{proof}

A single fixed starting time does not suffice for uniqueness of a full entrance law. For example, $T_0=R_\omega$ and $T_n=\id$ for $n\geq1$ give $B_{0,n}=R_\omega$ for all $n\geq1$, but $(\omega,\nu,\nu,\ldots)$ is an entrance law for every normal state $\nu$.

\begin{proposition}[Synchronization and convergence]\label{prop:anchor}
Let $P_n$ have pointwise loss of memory. Then $(P_n)$ is mixing if and only if $P_n\varphi_0$ converges in predual norm for one normal state $\varphi_0$. If loss of memory is uniform and $z_n=P_n\varphi_0$, then
\[
 \sup_{\varphi\in\Sn}\norm{P_n\varphi-\omega}
 \leq2\delta(P_n)+\norm{z_n-\omega}.
\]
The same statements hold for Ces\`aro averages. In particular, a common invariant normal state $\omega$ converts pointwise or uniform memory loss into the corresponding mixing property, in either direction; it also converts mean memory loss into mean ergodicity.
\end{proposition}
\begin{proof}
Use
$\norm{P_n\varphi-\omega}\leq\norm{P_n\varphi-P_n\varphi_0}+\norm{P_n\varphi_0-\omega}$.
The necessity is immediate. If every $T_n$ fixes $\omega$, then every product and every average fixes it, so it is an admissible constant anchor.
\end{proof}

\begin{example}[Complete memory loss with a moving output]\label{ex:replacement}
Choose distinct normal states $\alpha,\gamma$ and let $T_n=R_{\sigma_n}$, where $\sigma_n=\alpha$ for even $n$ and $\sigma_n=\gamma$ for odd $n$. Since
$R_\alpha R_\gamma=R_\alpha$, one has
\[
 F_{0,n}=R_{\sigma_{n-1}},\qquad B_{0,n}=R_{\sigma_0}\quad(n\geq1).
\]
Both image diameters vanish after one step. The backward process is stationary after that step; the forward process does not converge. Its Ces\`aro averages converge uniformly to $R_{(\alpha+\gamma)/2}$. This example works already on a two-state classical algebra. It separates synchronization from convergence, not finite from infinite dimension.

If the sequence $\sigma_n$ instead alternates between $\alpha$ and $\gamma$ on blocks whose successive lengths dominate the sum of all previous lengths, its Ces\`aro means fail to converge. The forward process then has even exponential loss of memory, and uniform mean loss of memory, without mean ergodicity.
\end{example}

Nesting is sufficient for convergence under uniform memory loss; it is not necessary. For instance, take distinct states $\alpha,\omega$ and $\sigma_n=(1-1/(n+2))\omega+(1/(n+2))\alpha$. The forward replacement products converge uniformly to $\omega$, but their singleton images are distinct and hence are not nested.

\section{Quantitative contraction, windows, and tracking}\label{sec:rates}

\begin{theorem}[Deterministic contraction estimates]\label{thm:rates}
For either order $P_{m,n}=F_{m,n}$ or $P_{m,n}=B_{m,n}$,
\begin{equation}\label{eq:product-rate}
 \delta(P_{m,n})\leq\prod_{j=m}^{n-1}\delta(T_j)
 \leq Q_{m,n}:=\prod_{j=m}^{n-1}(1-\beta(T_j)).
\end{equation}
Consequently, $\sum_{j=0}^\infty\beta(T_j)=\infty$ implies uniform loss of memory in both orders from every fixed starting time. Backward normal entrance laws then exist uniquely and satisfy
\[
 \sup_{\varphi\in\Sn}\norm{B_{m,n}\varphi-\omega_m}\leq2Q_{m,n}.
\]
If every $T_j$ fixes a common normal state $\omega$, then
\[
 \sup_{\varphi\in\Sn}\norm{P_{m,n}\varphi-\omega}\leq2Q_{m,n}
\]
in both orders.
\end{theorem}
\begin{proof}
Apply Lemma~\ref{lem:delta} repeatedly and then Theorem~\ref{thm:minorization}. The estimate $1-u\leq e^{-u}$ for $0\leq u\leq1$ gives $Q_{m,n}\leq\exp(-\sum_{j=m}^{n-1}\beta(T_j))$; a factor equal to zero only strengthens the bound. Lemma~\ref{lem:delta}, Theorem~\ref{thm:entrance}, and Proposition~\ref{prop:anchor} give the remaining assertions.
\end{proof}

If $\beta(T_j)\geq\varepsilon>0$ on a set $G$ of indices and $N_G(m,n)=\#(G\cap[m,n))$, the bound is
\begin{equation}\label{eq:good-count}
 2Q_{m,n}\leq2(1-\varepsilon)^{N_G(m,n)}.
\end{equation}
Positive lower density, namely $N_G(0,n)\geq cn-C_0$ with $c>0$, gives an exponential rate in physical time. Merely having $N_G(0,n)=O(n)$ does not: that upper bound is true for every subset of indices. If $N_G(0,n)\asymp\log n$, the guaranteed bound is polynomial. Furthermore, divergence of $\sum_j\beta(T_j)$ can hold when $\beta(T_j)\to0$, so a positive accumulation point is not needed.

Contraction may also occur only in blocks. For a partition $m=n_0<n_1<\cdots$ and either chosen direction, put $q_k=\delta(P_{n_k,n_{k+1}})$. Whenever $n\geq n_K$, composition yields
\[
 \delta(P_{m,n})\leq\prod_{k=0}^{K-1}q_k.
\]
This condition is strictly more flexible than a one-step minorization condition. The order of each block must agree with the order of the process; a forward block estimate cannot simply be transferred to a backward block.

\begin{proposition}[Uniformity over starting times]\label{prop:windows}
For either direction let
\[
 h(\ell)=\sup_{m\geq0}\delta(P_{m,m+\ell}).
\]
Then the following conditions are equivalent:
\begin{enumerate}
\item $h(\ell)\to0$ as $\ell\to\infty$.
\item There are $L\geq1$ and $q<1$ such that $h(L)\leq q$.
\end{enumerate}
Under (2), $h(\ell)\leq q^{\lfloor\ell/L\rfloor}$; if $q=0$, the coefficient is zero for all $\ell\geq L$.
\end{proposition}
\begin{proof}
The composition identities and submultiplicativity give $h(a+b)\leq h(a)h(b)$. If $h(\ell)\to0$, choose $L$ with $h(L)<1$. Conversely, write $\ell=kL+r$, $0\leq r<L$, and use $h(r)\leq1$.
\end{proof}

The statement ``for each fixed $m$, $\delta(P_{m,n})\to0$'' is weaker than uniformity over $m$. For example, take $T_j=\id$ except at $j=2^k$, where $T_j=\tfrac12\,\id+\frac12R_\omega$ on a nontrivial algebra. Every fixed tail contains infinitely many contraction steps, so it loses memory uniformly over states. There are nevertheless arbitrarily long intervals containing only identity maps, and $h(\ell)=1$ for every finite $\ell$.

\begin{theorem}[Tracking an approximate forward trajectory]\label{thm:tracking}
Suppose $\delta(F_{j,n})\leq q_{j,n}\leq1$ for $j<n$, and put $q_{n,n}=1$. For a sequence $\nu_j\in\Sn$, set $\varepsilon_j=\norm{T_j\nu_j-\nu_{j+1}}$. Then
\begin{equation}\label{eq:tracking}
 \norm{F_{m,n}\varphi-\nu_n}
 \leq2q_{m,n}+\sum_{j=m}^{n-1}q_{j+1,n}\varepsilon_j.
\end{equation}
If $q_{j,n}\to0$ for each fixed $j$ and $\sum_j\varepsilon_j<\infty$, the left-hand side tends to zero uniformly over initial states at any fixed $m$. The same conclusion holds when $q_{j,n}\leq Cr^{n-j}$ for some $r<1$ and $\varepsilon_j\to0$.
\end{theorem}
\begin{proof}
A telescoping expansion gives
\[
 F_{m,n}\varphi-\nu_n
 =F_{m,n}(\varphi-\nu_m)
  +\sum_{j=m}^{n-1}F_{j+1,n}(T_j\nu_j-\nu_{j+1}).
\]
Every functional in parentheses belongs to $E_0$. Apply the restriction norm bound and $\norm{\varphi-\nu_m}\leq2$. In the summable case, split the sum into a finite initial part and a tail with small total mass. In the exponential case, first make the late errors small and use the geometric sum bound, then let the finitely many early contributions decay.
\end{proof}

In particular, contraction gives a synchronized moving state, not necessarily a stationary limit. If $\nu_n\to\omega$, tracking yields mixing to $\omega$; if $\nu_n$ oscillates, there need not be any limiting state. This is the forward counterpart of the backward entrance-law statement.

\section{Infinite-dimensional counterexamples and universality}\label{sec:dimension}

The following retraction transfers classical counterexamples without imposing a trace or an algebra type.

\begin{lemma}[A countable classical retract]\label{lem:retract}
Every infinite-dimensional $\sigma$-finite von Neumann algebra admits normal unital completely positive maps
\[
 J:\ell^\infty(\N_0)\longrightarrow\M,\qquad
 E:\M\longrightarrow\ell^\infty(\N_0),\qquad EJ=\id,
\]
where $J$ is a unital injective $*$-homomorphism. Its predual $E_*:\ell^1(\N_0)\to\M_*$ is an isometry, $J_*E_*=\id$, and both preduals preserve normal states. If $Q_n$ are classical stochastic operators on $\ell^1(\N_0)$, then
\begin{equation}\label{eq:lift}
 \widehat T_n=E_*Q_nJ_*
\end{equation}
are normal quantum-channel preduals on $\M$. Every nonempty product of the lifted maps equals $E_*$ followed by the corresponding classical product followed by $J_*$. In particular, output distances and suprema over states for nonempty products are exactly the classical ones.
\end{lemma}
\begin{proof}
We first show that $\M$ contains an infinite family of pairwise orthogonal nonzero projections. Suppose not. Then every nonzero projection $e$ dominates a minimal projection of $\M$: otherwise one can choose a strictly decreasing sequence $e=e_1>e_2>\cdots$ of nonzero projections, and the differences $e_k-e_{k+1}$ form an infinite orthogonal family. Let $\{q_1,\ldots,q_r\}$ be a maximal orthogonal family of minimal projections; it is finite by assumption. If $\one-\sum_iq_i\ne0$, it dominates a further minimal projection, contradicting maximality. Hence $\sum_iq_i=\one$. For a minimal projection $q$, the von Neumann algebra $q\M q$ has no projections other than $0$ and $q$, so the spectral theorem gives $q\M q=\mathbb Cq$. Let $x,y\in q_i\M q_j$ with $x\ne0$. Then $x^*x\in q_j\M q_j$ and $xx^*\in q_i\M q_i$, so $x^*x=cq_j$ and $xx^*=cq_i$ with $c=\norm x^2>0$. Also $x^*y=\mu q_j$ for some $\mu\in\mathbb C$. Therefore $y=q_iy=c^{-1}xx^*y=c^{-1}\mu\,xq_j=c^{-1}\mu\,x$. Thus $\dim q_i\M q_j\leq1$, and $\dim\M=\dim\sum_{i,j}q_i\M q_j\leq r^2$, a contradiction.

Let $(q_j)_{j\geq0}$ be such an infinite orthogonal family. Put $p_0=q_0+(\one-\sum_jq_j)$ and $p_j=q_j$ for $j\geq1$, the sum being strongly convergent. These are nonzero pairwise orthogonal projections with $\sum_jp_j=\one$.

Choose a faithful normal state $\chi$ and set
\[
 \sigma_j(x)=\frac{\chi(p_jxp_j)}{\chi(p_j)},\qquad
 J(a)=\sum_{j\geq0}a_jp_j,\qquad E(x)=(\sigma_j(x))_{j\geq0}.
\]
These maps are normal and unital. The map $J$ is a $*$-homomorphism; $E$ is completely positive because each coordinate is a positive scalar functional. Since $\sigma_j(p_k)=\delta_{jk}$, $EJ=\id$.

For $a\in\ell^1$, $E_*a=\sum_ja_j\sigma_j$ in norm. The triangle inequality gives an upper bound $\sum_j|a_j|$, and testing on $J(b)$ with $\norm b_\infty\leq1$ and $a_jb_j=|a_j|$ gives the reverse bound. Thus $E_*$ is isometric. The predual identity follows from $EJ=\id$. Every classical stochastic map has a normal unital positive, hence completely positive, adjoint on the commutative algebra. Therefore the adjoint of~\eqref{eq:lift} is $JQ_n^*E$, a normal unital CP map. Cancellation of $J_*E_*$ proves the product formula. Finally $J_*$ maps states onto all probability vectors, since $J_*E_*=\id$, proving equality of suprema.
\end{proof}

\begin{example}[Pointwise mixing without uniformity]\label{ex:countdown}
On $\ell^1(\N_0)$, define the stochastic operator
\[
 C e_0=e_0,\qquad C e_j=e_{j-1}\quad(j\geq1),
\]
where $e_j$ is the point mass at $j$. Its adjoint on $\ell^\infty$ is the normal unital $*$-homomorphism
$C^*a=(a_0,a_0,a_1,a_2,\ldots)$.
For every probability vector $p$,
\begin{equation}\label{eq:countdown}
 \norm{C^np-e_0}_1=2\sum_{j>n}p_j\longrightarrow0.
\end{equation}
The convergence is not uniform: take $p=e_{n+1}$ to obtain distance two. In fact $\delta(C^n)=1$ for all $n$.

Let $A_N=N^{-1}\sum_{k=0}^{N-1}C^k$. Pointwise mixing gives $A_Np\to e_0$. But for $p=e_N$ all states $C^kp$, $0\leq k<N$, are supported outside zero, so $\norm{A_Np-e_0}_1=2$. Thus the homogeneous process is mixing and mean ergodic, but is neither uniformly mixing nor uniformly mean ergodic. Lifting $C$ by Lemma~\ref{lem:retract} gives the same homogeneous counterexample on every infinite-dimensional $\sigma$-finite $\M$, in both directions.
\end{example}

This example simultaneously proves that pointwise mixing does not imply uniform mean ergodicity and that mean ergodicity does not imply its uniform version. A trace on the algebra does not by itself repair either implication: the lifted construction works on any $\mathrm{II}_1$ factor, although the chosen channels need not preserve its trace.

\begin{example}[A strict prefix bound does not accumulate]\label{ex:prefix}
Fix $0<q<1$, take $T_0=q\,\id+(1-q)R_{e_0}$, and $T_j=C$ for $j\geq1$, with $C$ as above. These maps commute because $C e_0=e_0$. For $n\geq1$, both ordered products are
\[
 P_n=qC^{n-1}+(1-q)R_{e_0}.
\]
Hence every initial state converges to $e_0$, while
\[
 \sup_p\norm{P_np-e_0}_1=2q,\qquad \delta(P_n)=q.
\]
For the coefficient identity use the upper bound and the pair $e_0,e_n$, whose outputs have distance $2q$. Thus even mixing together with $\limsup_n\delta(P_n)<1$ does not imply uniform mixing. Hence a strict bound on the coefficient of the whole prefix cannot replace a recurrent block-contraction hypothesis. Again the example lifts to every infinite-dimensional $\sigma$-finite algebra.
\end{example}

\begin{theorem}[Compact backward images without stability]\label{thm:reservoir}
Let $(a_n)_{n\geq1}$ be any sequence in $\{0,1\}$. There are stochastic operators $Q_n$ on $\ell^1(\N_0)$ such that, for $R_n=Q_1\cdots Q_n$,
\begin{align}
 R_n\bigl(\mathcal P(\N_0)\bigr)&=\operatorname{conv}\{e_0,e_1\},\label{eq:reservoir-image}\\
 \norm{R_np-e_{a_n}}_1&\leq2\sum_{j>n}p_j,\label{eq:reservoir-anchor}\\
 \norm{R_np-R_nq}_1&\leq2\sum_{j>n}(p_j+q_j).\label{eq:reservoir-ml}
\end{align}
Consequently the backward products have pointwise loss of memory for every choice of $(a_n)$, but they converge to one state if and only if $(e_{a_n})$ converges. Their Ces\`aro averages converge to one state if and only if the scalar means $N^{-1}\sum_{n=1}^Na_n$ converge. Uniform loss of memory always fails: $\delta(R_n)=1$.
\end{theorem}
\begin{proof}
For each $n\geq1$ define a deterministic function $r_n:\N_0\to\{0,1\}$ by
\[
 r_n(j)=\begin{cases}a_n,&j\leq n,\\ j\bmod2,&j>n.\end{cases}
\]
Let $s_1=r_1$. Given $n\geq1$, put $u_n=2n+2$, $v_n=2n+3$, and define
\[
 s_{n+1}(j)=\begin{cases}
 u_n,&r_{n+1}(j)=0,\\v_n,&r_{n+1}(j)=1.
 \end{cases}
\]
Since $r_n(u_n)=0$ and $r_n(v_n)=1$, one has
$r_n\circ s_{n+1}=r_{n+1}$. Thus induction gives
$s_1\circ\cdots\circ s_n=r_n$. Define $Q_ne_j=e_{s_n(j)}$. Each $Q_n$ is stochastic, with normal unital $*$-homomorphic adjoint $a\mapsto a\circ s_n$. The backward product is therefore
\[
 R_ne_j=e_{r_n(j)}.
\]
Both parities occur among $j>n$, so the image of all probability vectors is exactly the segment in~\eqref{eq:reservoir-image}. Its diameter is two, giving $\delta(R_n)=1$. All coordinates $j\leq n$ are sent to $a_n$. Only the tail mass can differ from $e_{a_n}$, and its contribution to an $\ell^1$ difference is at most twice that mass. This proves~\eqref{eq:reservoir-anchor}; the triangle inequality proves~\eqref{eq:reservoir-ml}.

For each fixed $p$, the right-hand side of~\eqref{eq:reservoir-anchor} tends to zero. Taking $p=e_0$ gives $R_ne_0=e_{a_n}$ exactly, proving the assertion about ordinary limits. For averages,
\[
 \left\lVert\frac1N\sum_{n=1}^NR_np-
                    \frac1N\sum_{n=1}^Ne_{a_n}\right\rVert_1
 \leq\frac2N\sum_{n=1}^N\sum_{j>n}p_j\longrightarrow0
\]
by Ces\`aro summation. The last vector average is
$(1-\bar a_N)e_0+\bar a_Ne_1$, where $\bar a_N=N^{-1}\sum_{n=1}^Na_n$. This proves the Ces\`aro assertion. Including an initial identity term or shifting the average by one term changes the average by at most $2/N$ and has no effect on the conclusions.
\end{proof}

For $a_n=n\bmod2$, this gives pointwise backward memory loss and mean ergodicity to $(e_0+e_1)/2$, but not mixing. Uniform mean ergodicity also fails: for $j>N$, all $R_ne_j$, $1\leq n\leq N$, equal $e_{j\bmod2}$, at distance one from the mean limit.

For a failure of mean ergodicity, let $N_0=0$, $N_k=2^{2^k}$ for $k\geq1$, and prescribe $a_n=k\bmod2$ for $N_{k-1}<n\leq N_k$. Then $N_{k-1}/N_k\to0$, and the means at times $N_k$ approach the alternating block values. Hence the scalar Ces\`aro means do not converge. Pointwise mean loss of memory still holds by~\eqref{eq:reservoir-ml} and Ces\`aro summation. This example therefore separates weak mean memory loss from mean ergodicity even for backward products.

\begin{proof}[Proof of Theorem~\ref{thm:intro-dimension}]
If $\M$ is finite dimensional, its normal state space is norm compact. Proposition~\ref{prop:hierarchy} upgrades pointwise backward memory loss to uniform memory loss, and Theorem~\ref{thm:intro-uniform} gives mixing. Hence (1) implies (2), which trivially implies (3).

If $\M$ is infinite dimensional, lift the stochastic sequence in Theorem~\ref{thm:reservoir} using Lemma~\ref{lem:retract}. With $\sigma_i=E_*e_i$, every nonempty backward product has state image exactly $\operatorname{conv}\{\sigma_0,\sigma_1\}$, a norm-compact segment. Equations~\eqref{eq:reservoir-anchor}--\eqref{eq:reservoir-ml} and the isometry give pointwise loss of memory on all normal states of $\M$. Choose the non-Ces\`aro-convergent block sequence above. The state $\sigma_0$ has output $\sigma_{a_n}$, so its trajectory and its Ces\`aro averages do not converge. Thus (3) fails. This proves (3) implies (1) and the additional assertion.
\end{proof}

\begin{remark}[What compactness does and does not repair]
The reservoir example has finite-rank products, a collectively norm-compact set of all nontrivial state outputs, and nested images. None of these assumptions converts pointwise backward memory loss into uniform memory loss. The missing compactness is on the input-state space: optimizing input states may depend on $n$ and may put all their mass in an unexplored tail. In contrast, once the image diameters actually tend to zero, completeness alone proves Theorem~\ref{thm:nested}.
\end{remark}

\section{Rates, periodicity, and homogeneous comparisons}\label{sec:homogeneous}

\begin{example}[Arbitrarily slow uniform mixing]\label{ex:slow}
On any nontrivial von Neumann algebra choose $\omega\in\Sn$ and a positive nonincreasing sequence $b_n\downarrow0$, with $b_0=1$. Define
\[
 T_n=\frac{b_{n+1}}{b_n}\,\id+
          \left(1-\frac{b_{n+1}}{b_n}\right)R_\omega.
\]
The maps commute, and both ordered products satisfy
\[
 P_{0,n}=b_n\,\id+(1-b_n)R_\omega,\qquad
 \delta(P_{0,n})=b_n,
\]
while
\[
 \sup_{\varphi\in\Sn}\norm{P_{0,n}\varphi-\omega}
 =b_n c_\omega,\qquad
 c_\omega:=\sup_{\varphi\in\Sn}\norm{\varphi-\omega}>0.
\]
The coefficient equality follows from $P_{0,n}|_{E_0}=b_n\,\id$. The channel can therefore have any prescribed decreasing positive relaxation profile. Choosing $b_n=1/(n+1)$ proves that uniform mixing need not be exponential for inhomogeneous processes. Choosing $b_n=1/\log(n+e)$, for which $b_0=1$, gives a still slower example.

For $M_d(\mathbb C)$ and $\omega=I/d$, the correct constant is $c_\omega=2(1-1/d)$, not two. Indeed, diagonalize $\rho$ and write $\rho=\sum_j\lambda_jP_j$; convexity bounds $\norm{\rho-I/d}_1$ by the pure-state value $2(1-1/d)$, and every pure state attains that value. In particular no state can be orthogonal to the full support of $I/d$.
\end{example}

\begin{example}[Uniform mean ergodicity without mixing]\label{ex:periodic}
On the two-state algebra $\mathbb C^2$, let $Ue_0=e_1$, $Ue_1=e_0$. Its invariant state is $\omega=(e_0+e_1)/2$ and
\[
 \sup_p\left\lVert\frac1N\sum_{k=0}^{N-1}U^kp-\omega\right\rVert_1
 \leq\frac1N.
\]
For even $N$ the average is exactly $\omega$; for odd $N$ it differs by $(p-\omega)/N$. Nevertheless $U^ne_0$ alternates and does not converge, and $\norm{U^ne_0-U^ne_1}_1=2$ for every $n$. Thus uniform mean ergodicity, and even uniform mean loss of memory, do not imply mixing or pointwise memory loss.

On $M_2(\mathbb C)$ the channel
\[
 T(\rho)=\begin{pmatrix}\rho_{11}&0\\0&\rho_{00}\end{pmatrix}
\]
(with indices $0,1$) gives the same example after deleting off-diagonal entries at the first step. It is completely positive: use Kraus operators $|1\rangle\langle0|$ and $|0\rangle\langle1|$. It has invariant density $I/2$. A two-projection version of Lemma~\ref{lem:retract}, with normal states $\sigma_j$ satisfying $\sigma_j(p_j)=1$, lifts the classical example to every von Neumann algebra with two nonzero orthogonal projections summing to one.
\end{example}

\begin{theorem}[Homogeneous benchmark]\label{thm:homogeneous}
Let $T=\Phi_*$ be a single normal channel predual.
\begin{enumerate}
\item Uniform mixing of $(T^n)$, uniform loss of memory, and the existence of $k\geq1$ with $\delta(T^k)<1$ are equivalent. Whenever they hold there is a unique invariant normal state $\omega$, and the convergence is exponential.
\item If an invariant normal state $\omega$ is given, pointwise loss of memory is equivalent to mixing to $\omega$; mean loss of memory is equivalent to mean ergodicity to $\omega$.
\item Mean ergodicity to one normal state is equivalent to the conjunction
\[
 \Sn\cap\Fix(T)\ne\varnothing,\qquad \Fix(\Phi)=\mathbb C\one.
\]
\end{enumerate}
\end{theorem}
\begin{proof}
Uniform mixing implies uniform memory loss. If $\delta(T^n)\to0$, some power has coefficient less than one. Conversely, if $q=\delta(T^k)<1$, then
$\delta(T^n)\leq q^{\lfloor n/k\rfloor}$.
The state images $T^n(\Sn)$ are nested, so Theorem~\ref{thm:nested} gives a normal state $\omega$ with
\[
 \sup_{\varphi\in\Sn}\norm{T^n\varphi-\omega}
 \leq2q^{\lfloor n/k\rfloor}.
\]
If $q>0$, this is an exponential bound with rate $q^{1/k}$ and prefactor $2/q$; if $q=0$, the products are rank one from time $k$ onward and an exponential estimate with any chosen rate follows after adjusting the finite prefactor. Passing to the limit in $T^{n+1}\varphi=T(T^n\varphi)$ proves $T\omega=\omega$. Any invariant normal state must equal this limit. This proves (1). Part (2) is Proposition~\ref{prop:anchor}.

For (3), suppose $A_Nf\to f(\one)\omega$ on states and hence on the complex linear span of states, which is all of $\M_*$. Since $TA_N-A_N=(T^N-\id)/N$, $T\omega=\omega$. If $\Phi x=x$, then
$\varphi(x)=(A_N\varphi)(x)\to\omega(x)$ for every normal state, so $x=\omega(x)\one$.

Conversely, let $T\omega=\omega$ and $\Fix(\Phi)=\mathbb C\one$. Hahn--Banach annihilator duality gives
\[
 \overline{\ran(\id-T)}^{\norm\cdot}
 =\{f\in\M_*: f(x)=0\ \text{for all }x\in\Fix(\Phi)\}
 =\{f\in\M_*:f(\one)=0\}.
\]
On $\ran(\id-T)$, telescoping gives
$A_N(\id-T)g=(g-T^Ng)/N\to0$. Since $\norm{A_N}\leq1$, this extends to the closure. Decomposing $f=f(\one)\omega+(f-f(\one)\omega)$ gives $A_Nf\to f(\one)\omega$.
\end{proof}

The uniform/exponential part is a standard Dobrushin consequence, recorded here for comparison with \cite{Muk13,Muk15}. It cannot be claimed as a new homogeneous theorem. Its relevance is that an inhomogeneous prefix bound less than one has no corresponding submultiplicative-power mechanism, as Example~\ref{ex:prefix} demonstrates.

\begin{example}[Unique invariant normal state is not enough]\label{ex:unique-invariant}
On $\ell^1(\N_0)$ define $Ve_0=e_0$ and $Ve_j=e_{j+1}$ for $j\geq1$. The only invariant probability vector is $e_0$: invariance first forces the mass at coordinate one to be zero, and then recursively forces all positive coordinates to have zero mass. But
\[
 \frac1N\sum_{k=0}^{N-1}V^ke_1
 =\frac1N\sum_{j=1}^Ne_j
\]
has no $\ell^1$ limit. Every coordinate tends to zero while its norm remains one. Thus uniqueness of a normal invariant state does not imply mean ergodicity in infinite dimension. Here $\Fix(V^*)$ consists of bounded functions constant on $\{1,2,\ldots\}$ with an independent value at zero; it is two dimensional, in agreement with Theorem~\ref{thm:homogeneous}(3). This example also lifts to any infinite-dimensional $\sigma$-finite $\M$: every invariant state of the lifted predual lies in $E_*\ell^1$, where the assertion reduces to the classical one.
\end{example}

\section{State-preserving Hilbert-space criteria}\label{sec:L2}

The counterexamples above do not claim that all distinctions persist under every additional state-preserving hypothesis. A useful positive criterion is available without a trace, and it illustrates what a common faithful normal state adds.

\begin{proposition}[GNS contraction criterion]\label{prop:GNS}
Suppose $\omega$ is a faithful normal state with $\omega\Phi_n=\omega$ for every $n$. On the GNS Hilbert space $L^2(\M,\omega)$, with cyclic vector $\Omega$, define
\[
 V_n(x\Omega)=\Phi_n(x)\Omega.
\]
Then $V_n$ extends to a contraction, fixes $\Omega$, and preserves $\Omega^\perp$. If there are $0\leq c_n\leq1$ such that
$\norm{V_n|_{\Omega^\perp}}\leq c_n$, then for either observable product $\Psi_{m,n}$, in either order,
\begin{equation}\label{eq:GNS}
 \norm{(\Psi_{m,n}(x)-\omega(x)\one)\Omega}
 \leq\left(\prod_{j=m}^{n-1}c_j\right)
       \norm{(x-\omega(x)\one)\Omega}.
\end{equation}
If $\prod_{j=m}^{n-1}c_j\to0$ for every fixed $m$, both corresponding state products are mixing in predual norm to $\omega$. The conclusion need not be uniform over all normal initial states.
\end{proposition}
\begin{proof}
Kadison--Schwarz and invariance give
\[
 \norm{\Phi_n(x)\Omega}^2
 =\omega(\Phi_n(x)^*\Phi_n(x))
 \leq\omega(\Phi_n(x^*x))=\omega(x^*x).
\]
Also $\langle\Omega,V_nx\Omega\rangle=\omega(x)$, so $V_n$ fixes $\Omega$ and preserves its orthogonal complement. Iteration proves~\eqref{eq:GNS} for either order.

It remains to justify the predual conclusion, which is stronger than mere convergence of each fixed correlation. In this faithful normal GNS representation, $\Omega$ is cyclic and separating for $\M$ and cyclic for its commutant $\M'$. The linear span of functionals
\[
 f_a(x)=\langle a'\Omega,x\Omega\rangle,\qquad a'\in\M',
\]
is norm dense in $\M_*$. Here is a direct verification. Vector functionals span a norm-dense subspace of the predual of a faithful normal representation. Approximate each pair of vectors by $b'\Omega,c'\Omega$, using cyclicity for $\M'$. The corresponding coefficient is
$\langle b'\Omega,xc'\Omega\rangle=\langle c'^*b'\Omega,x\Omega\rangle$,
with the inner product taken linear in the second variable. Approximation in vector norm implies approximation in functional norm, proving the density.

For $f_a$, Cauchy--Schwarz and~\eqref{eq:GNS} give
\[
 \bigl|f_a(\Psi_{m,n}(x))-f_a(\one)\omega(x)\bigr|
 \leq\norm{a'\Omega}\left(\prod_{j=m}^{n-1}c_j\right)
       \norm{(x-\omega(x)\one)\Omega}
 \leq2\norm{a'\Omega}\norm x\prod_{j=m}^{n-1}c_j.
\]
Taking the supremum over $\norm x\leq1$ proves predual convergence for these functionals. The channel preduals and $R_\omega$ are contractions, so norm density extends convergence to every normal functional, in particular every normal state.
\end{proof}

This proposition requires a quantitative contraction of the GNS implementation on the centered space. Faithfulness or existence of a trace alone does not give such a contraction. No identification $\M\simeq\M_*$, no tracial cycling of factors, and no use of a density relative to $\omega$ has entered the proof. The estimate in~\eqref{eq:GNS} is an $L^2$ estimate on observables; its conversion to a predual estimate uses a dense class of functionals whose constants are not uniform over the whole state space.

\section{Inhomogeneous transfer states over infinite bond algebras}\label{sec:FCS}

This section proves Theorem~\ref{thm:intro-FCS}. The physical algebra at each site is $M_d=M_d(\mathbb C)$, and $\mathcal A_{[p,q]}=\bigotimes_{k=p}^qM_d$ for $p\leq q$. For $p\leq q<q'$, the algebra $\mathcal A_{[p,q]}$ is identified with $\mathcal A_{[p,q]}\otimes\one\subseteq\mathcal A_{[p,q']}$. The quasi-local algebra $\mathcal A$ is the norm closure of $\bigcup_N\mathcal A_{[1,N]}$. The bond algebra $\mathcal B$ is an arbitrary von Neumann algebra. Since $M_d$ is finite dimensional, $M_d\otimes\mathcal B=M_d(\mathcal B)$ is a von Neumann algebra.

\subsection{Transfer maps and finite-volume states}

Fix normal unital completely positive maps $\mathcal E_k:M_d\otimes\mathcal B\to\mathcal B$, $k\geq1$, and put
\[
 \Phi_k(b)=\mathcal E_k(\one\otimes b),\qquad T_k=(\Phi_k)_*,\qquad B_{m,n}=T_m\cdots T_{n-1}\quad(1\leq m\leq n),
\]
with $B_{m,m}=\id$. Each $\Phi_k$ is a normal unital completely positive map on $\mathcal B$. For $p\leq q$ define $\mathcal E_{[p,q]}:\mathcal A_{[p,q]}\otimes\mathcal B\to\mathcal B$ recursively by $\mathcal E_{[p,p]}=\mathcal E_p$ and
\begin{equation}\label{eq:block-transfer-right}
 \mathcal E_{[p,q]}(x\otimes y\otimes b)=\mathcal E_q\bigl(y\otimes\mathcal E_{[p,q-1]}(x\otimes b)\bigr),
 \qquad x\in\mathcal A_{[p,q-1]},\ y\in M_d,\ b\in\mathcal B,
\end{equation}
extended by linearity. Thus $\mathcal E_{[p,q]}=\mathcal E_q\circ(\id_{M_d}\otimes\mathcal E_{[p,q-1]})\circ\theta$, where $\theta:\mathcal A_{[p,q-1]}\otimes M_d\otimes\mathcal B\to M_d\otimes\mathcal A_{[p,q-1]}\otimes\mathcal B$ is the normal $*$-isomorphism that moves the factor $M_d$ to the left. The amplification $\id_{M_d}\otimes\mathcal E_{[p,q-1]}$ of a normal unital completely positive map to $M_d(\,\cdot\,)$ is again normal, unital and completely positive. By induction, every $\mathcal E_{[p,q]}$ has these three properties. In this recursion site $q$ is outermost, and the bond variable $b$ enters at site $p$.

For $x\in\mathcal A_{[p,q]}$ write
\[
 \mathcal E_{[p,q]}^{x}(b)=\mathcal E_{[p,q]}(x\otimes b),\qquad
 y_x=\mathcal E_{[1,m]}(x\otimes\one)\quad(x\in\mathcal A_{[1,m]}).
\]
A unital completely positive map is contractive and $\norm{x\otimes b}=\norm x\norm b$. Hence $\mathcal E^x_{[p,q]}$ is a normal map with $\norm{\mathcal E^x_{[p,q]}}\leq\norm x$, and $\norm{y_x}\leq\norm x$. If $x=x^*$, then $\mathcal E^x_{[p,q]}$ preserves adjoints.

\begin{definition}
For $N\geq1$ and $\sigma\in\mathcal S_{\mathrm n}(\mathcal B)$, the \emph{finite-volume state with right boundary condition $\sigma$} is
\[
 \varphi_N^\sigma(x)=\sigma\bigl(\mathcal E_{[1,N]}(x\otimes\one)\bigr),\qquad x\in\mathcal A_{[1,N]}.
\]
\end{definition}

The functional $\varphi_N^\sigma$ is a state on $\mathcal A_{[1,N]}$, since $x\mapsto\mathcal E_{[1,N]}(x\otimes\one)$ is unital and positive. The left end of the chain is closed by the unit of $\mathcal B$, and the boundary state acts after the transfer map of site $N$. This is the convention of \cite{Sou26}, in which channels are trace preserving in the Schr\"odinger picture and therefore unital in the Heisenberg picture.

\begin{lemma}[Block factorisation]\label{lem:local-restriction}
For $p\leq s<q$, $x\in\mathcal A_{[p,s]}$, $w\in\mathcal A_{[s+1,q]}$ and $b\in\mathcal B$,
\begin{equation}\label{eq:assoc-right}
 \mathcal E_{[p,q]}(x\otimes w\otimes b)=\mathcal E^{w}_{[s+1,q]}\bigl(\mathcal E_{[p,s]}(x\otimes b)\bigr),
 \qquad \mathcal E^{\one}_{[s+1,q]}=\Phi_q\circ\cdots\circ\Phi_{s+1}.
\end{equation}
Consequently, for $1\leq m<N$, $x\in\mathcal A_{[1,m]}$ and $\sigma\in\mathcal S_{\mathrm n}(\mathcal B)$,
\begin{equation}\label{eq:local-restriction}
 \varphi_N^\sigma(x)=\bigl(B_{m+1,N+1}\sigma\bigr)(y_x),\qquad y_{x\otimes\one}=\Phi_{m+1}(y_x).
\end{equation}
If moreover $z\in\mathcal A_{[m+\ell+1,n]}$ with $\ell\geq1$ and $n>m+\ell$, then for every $\nu\in\mathcal S_{\mathrm n}(\mathcal B)$,
\begin{equation}\label{eq:two-block}
 \nu\bigl(\mathcal E_{[1,n]}(x\otimes\one_\ell\otimes z\otimes\one)\bigr)
 =\bigl(B_{m+1,m+\ell+1}\,g\bigr)(y_x),\qquad g=\nu\circ\mathcal E^{z}_{[m+\ell+1,n]},
\end{equation}
where $\one_\ell$ denotes the unit of $\mathcal A_{[m+1,m+\ell]}$.
\end{lemma}
\begin{proof}
By linearity it suffices to take $w=w'\otimes y$ with $w'\in\mathcal A_{[s+1,q-1]}$ and $y\in M_d$ when $q>s+1$. For $q=s+1$, the first identity in~\eqref{eq:assoc-right} is~\eqref{eq:block-transfer-right}. Assume it for $q-1$. Then
\[
 \mathcal E_{[p,q]}(x\otimes w'\otimes y\otimes b)
 =\mathcal E_q\bigl(y\otimes\mathcal E_{[p,q-1]}(x\otimes w'\otimes b)\bigr)
 =\mathcal E_q\bigl(y\otimes\mathcal E^{w'}_{[s+1,q-1]}(\mathcal E_{[p,s]}(x\otimes b))\bigr),
\]
and the right-hand side is $\mathcal E^{w'\otimes y}_{[s+1,q]}(\mathcal E_{[p,s]}(x\otimes b))$ by~\eqref{eq:block-transfer-right} applied on $[s+1,q]$. The formula for $\mathcal E^{\one}_{[s+1,q]}$ follows by the same induction, since $\mathcal E_q(\one\otimes b)=\Phi_q(b)$.

For~\eqref{eq:local-restriction}, apply~\eqref{eq:assoc-right} with $p=1$, $s=m$, $q=N$, $w=\one$ and $b=\one$. This gives $\mathcal E_{[1,N]}(x\otimes\one\otimes\one)=\Phi_N\circ\cdots\circ\Phi_{m+1}(y_x)$. The predual of $\Phi_N\circ\cdots\circ\Phi_{m+1}$ is $T_{m+1}\cdots T_N=B_{m+1,N+1}$. The case $N=m+1$ gives $y_{x\otimes\one}=\Phi_{m+1}(y_x)$. For~\eqref{eq:two-block}, apply~\eqref{eq:assoc-right} first with $s=m+\ell$, $w=z$, and then with $s=m$, $w=\one_\ell$. This gives
\[
 \mathcal E_{[1,n]}(x\otimes\one_\ell\otimes z\otimes\one)
 =\mathcal E^z_{[m+\ell+1,n]}\bigl(\Phi_{m+\ell}\circ\cdots\circ\Phi_{m+1}(y_x)\bigr).
\]
Applying $\nu$ and passing to preduals completes the proof.
\end{proof}

\subsection{Thermodynamic limit and decay of correlations}

\begin{proof}[Proof of Theorem~\ref{thm:intro-FCS}]
By Theorem~\ref{thm:entrance}, applied to the sequence $(T_k)_{k\geq1}$, there is a unique family $(\omega_k)_{k\geq1}$ of normal states of $\mathcal B$ with $\omega_k=T_k\omega_{k+1}$, and
\begin{equation}\label{eq:FCS-entrance}
 \sup_{\sigma\in\mathcal S_{\mathrm n}(\mathcal B)}\norm{B_{k,N+1}\sigma-\omega_k}\leq2\,\delta(B_{k,N+1}).
\end{equation}
For $x\in\mathcal A_{[1,m]}$ define $\varphi_\infty(x)=\omega_{m+1}(y_x)$. This is consistent: if $x\in\mathcal A_{[1,m]}$ is regarded as $x\otimes\one\in\mathcal A_{[1,m+1]}$, then $y_{x\otimes\one}=\Phi_{m+1}(y_x)$ by Lemma~\ref{lem:local-restriction}. Hence
\[
 \omega_{m+2}\bigl(\Phi_{m+1}(y_x)\bigr)=(T_{m+1}\omega_{m+2})(y_x)=\omega_{m+1}(y_x).
\]
Thus $\varphi_\infty$ is a well-defined linear functional on $\bigcup_m\mathcal A_{[1,m]}$. It is positive and unital on each $\mathcal A_{[1,m]}$, because $x\mapsto y_x$ is unital and positive. It therefore has norm one and extends uniquely to a state on $\mathcal A$.

By~\eqref{eq:local-restriction} and~\eqref{eq:FCS-entrance}, for $x\in\mathcal A_{[1,m]}$,
\[
 |\varphi_N^\sigma(x)-\varphi_\infty(x)|
 =\bigl|(B_{m+1,N+1}\sigma-\omega_{m+1})(y_x)\bigr|
 \leq2\,\delta(B_{m+1,N+1})\norm x.
\]
This is the first estimate, since $B_{m+1,N+1}=T_{m+1}\cdots T_N$. Uniqueness follows because the right-hand side tends to zero for every fixed $m$: a state with the same property has the same restriction to every $\mathcal A_{[1,m]}$, and these restrictions determine a state on $\mathcal A$.

For the clustering bound, let $n=m+\ell+r$ and put $g=\omega_{n+1}\circ\mathcal E^z_{[m+\ell+1,n]}$, a normal functional with $\norm g\leq\norm z$. By definition of $\varphi_\infty$ and~\eqref{eq:two-block},
\[
 \varphi_\infty(x\otimes\one_\ell\otimes z)=\bigl(B_{m+1,m+\ell+1}g\bigr)(y_x),
 \qquad \varphi_\infty(z)=g(\one),
\]
where the second identity is~\eqref{eq:two-block} with $x=\one$, since $y_{\one}=\one$ and $(B_{m+1,m+\ell+1}g)(\one)=g(\one)$. Also $\varphi_\infty(x)=\omega_{m+1}(y_x)=(B_{m+1,m+\ell+1}\omega_{m+\ell+1})(y_x)$. Hence
\[
 \varphi_\infty(x\otimes\one_\ell\otimes z)-\varphi_\infty(x)\varphi_\infty(z)
 =\bigl(B_{m+1,m+\ell+1}h\bigr)(y_x),\qquad h=g-g(\one)\,\omega_{m+\ell+1}.
\]
Suppose first that $z=z^*$. Then $g$ is self-adjoint, $h(\one)=0$, and $\norm h\leq2\norm g\leq2\norm z$. Therefore $h\in E_0(\mathcal B)$ and
\[
 \bigl|(B_{m+1,m+\ell+1}h)(y_x)\bigr|\leq\delta(B_{m+1,m+\ell+1})\norm h\norm{y_x}
 \leq2\,\delta(B_{m+1,m+\ell+1})\norm x\norm z.
\]
For general $z$, write $z=z_1+iz_2$ with $z_j$ self-adjoint and $\norm{z_j}\leq\norm z$, and use linearity in $z$. This gives the factor $4$. Finally, $B_{m+1,m+\ell+1}=T_{m+1}\cdots T_{m+\ell}$.
\end{proof}

\begin{corollary}[Quantitative regimes]\label{cor:FCS-rates}
Under the notation of Theorem~\ref{thm:intro-FCS}, put $\beta_k=\beta(T_k)$.
\begin{enumerate}
\item If $\sum_k\beta_k=\infty$, the hypothesis of Theorem~\ref{thm:intro-FCS} holds, with
\[
 \bigl\|\varphi_N^\sigma|_{\mathcal A_{[1,m]}}-\varphi_\infty|_{\mathcal A_{[1,m]}}\bigr\|\leq2\prod_{k=m+1}^{N}(1-\beta_k).
\]
\item If there are $L\geq1$ and $q<1$ with $\sup_{m\geq0}\delta(T_{m+1}\cdots T_{m+L})\leq q$, then boundary dependence and correlations decay exponentially, uniformly in position:
\[
 \bigl\|\varphi_N^\sigma|_{\mathcal A_{[1,m]}}-\varphi_\infty|_{\mathcal A_{[1,m]}}\bigr\|\leq2\,q^{\lfloor(N-m)/L\rfloor},
\]
\[
 \bigl|\varphi_\infty(x\otimes\one_\ell\otimes z)-\varphi_\infty(x)\varphi_\infty(z)\bigr|\leq4\,q^{\lfloor\ell/L\rfloor}\norm x\norm z.
\]
\end{enumerate}
\end{corollary}
\begin{proof}
Part (1) follows from Theorem~\ref{thm:intro-FCS} and~\eqref{eq:product-rate}, since $\delta(T_{m+1}\cdots T_N)\leq\prod_{k=m+1}^N(1-\beta_k)$ and the product tends to zero for every fixed $m$ when $\sum_k\beta_k=\infty$. For part (2), apply Proposition~\ref{prop:windows} to the backward products of $(T_k)_{k\geq1}$: the window coefficient satisfies $h(\ell)\leq q^{\lfloor\ell/L\rfloor}$, and both estimates of Theorem~\ref{thm:intro-FCS} are bounded by $h$ evaluated at the length of the relevant product.
\end{proof}

\subsection{Connection with matrix product states}

Let $\mathcal B=M_D$ and let $K^{[k]}_1,\ldots,K^{[k]}_d\in M_D$ satisfy the left-canonical condition $\sum_iK^{[k]\dagger}_iK^{[k]}_i=\one$. Define the isometry $V_k:\mathbb C^D\to\mathbb C^d\otimes\mathbb C^D$ by $V_k\xi=\sum_j|j\rangle\otimes K^{[k]}_j\xi$, and set
\[
 \mathcal E_k(y\otimes b)=V_k^\dagger(y\otimes b)V_k
 =\sum_{i,j=1}^d\langle i|y|j\rangle\,K^{[k]\dagger}_i\,b\,K^{[k]}_j.
\]
Then $\mathcal E_k$ is unital and completely positive, and $\Phi_k(b)=\sum_iK^{[k]\dagger}_ibK^{[k]}_i$ is the Heisenberg dual of the trace-preserving channel $\rho\mapsto\sum_iK^{[k]}_i\rho K^{[k]\dagger}_i$. Iterating~\eqref{eq:block-transfer-right} gives, for $b=\one$,
\[
 \mathcal E_{[1,N]}\bigl(|i_1\cdots i_N\rangle\langle j_1\cdots j_N|\otimes\one\bigr)
 =K^{[N]\dagger}_{i_N}\cdots K^{[1]\dagger}_{i_1}K^{[1]}_{j_1}\cdots K^{[N]}_{j_N}.
\]
For a boundary density $\sigma$, therefore,
\[
 \varphi^\sigma_N\bigl(|i_1\cdots i_N\rangle\langle j_1\cdots j_N|\bigr)
 =\Tr\bigl(K^{[1]}_{j_1}\cdots K^{[N]}_{j_N}\,\sigma\,K^{[N]\dagger}_{i_N}\cdots K^{[1]\dagger}_{i_1}\bigr).
\]
This is the finite-volume matrix product expression with right boundary $\sigma$. The limiting formula $\varphi_\infty(x)=\omega_{m+1}(y_x)$ is the boundary-state representation of \cite[Theorem~1.3]{Sou26}, with $(\omega_k)$ the backward entrance law. In this finite-dimensional case, Theorem~\ref{thm:minorization} reduces the minorization hypothesis to the Markov--Dobrushin condition of \cite{Sou26,SB25}, and the coefficients $\delta(T_{m+1}\cdots T_N)$ are the product trace-Dobrushin coefficients of \cite{P26}.

\begin{example}[An infinite bond algebra]\label{ex:infinite-bond}
Let $K$ be an infinite-dimensional separable Hilbert space and $\mathcal B=B(K)$. Fix unitaries $U_k:K\to\mathbb C^d\otimes K$, a normal state $\nu$ of $\mathcal B$, states $\psi_k$ of $M_d$, and numbers $p_k\in[0,1]$. Define
\[
 \mathcal E_k(y\otimes b)=(1-p_k)\,U_k^*(y\otimes b)U_k+p_k\,\psi_k(y)\,\nu(b)\,\one.
\]
Each $\mathcal E_k$ is normal, unital and completely positive. Moreover
\[
 T_k=(1-p_k)T_k'+p_kR_\nu,\qquad T_k'=\bigl(U_k^*(\one\otimes\,\cdot\,)U_k\bigr)_*,
\]
so $T_k\varphi\geq p_k\nu$ for every normal state $\varphi$, and $\beta(T_k)\geq p_k$. If $\sum_kp_k=\infty$, Corollary~\ref{cor:FCS-rates}(1) applies. If $p_k\geq p>0$ for all $k$, it gives exponential decay of boundary dependence and of correlations, with $L=1$ and $q=1-p$. Here the bond algebra is a type $\mathrm I_\infty$ factor. More generally, replace $U_k^*(\,\cdot\,)U_k$ by a normal $*$-isomorphism $\pi_k:M_d\otimes\mathcal B\to\mathcal B$. Such isomorphisms exist whenever $\mathcal B$ is properly infinite, in particular for every type $\mathrm{III}$ factor, and when $\mathcal B$ is the hyperfinite $\mathrm{II}_1$ factor, whose fundamental group is $\R_+^*$. The same conclusions hold verbatim.
\end{example}

\begin{remark}[Uniformity is necessary]\label{rem:FCS-pointwise}
The estimate in Theorem~\ref{thm:intro-FCS} is uniform in the boundary state. The following example shows that pointwise mixing of the backward products does not suffice for boundary independence of the local states. Let $d\geq2$, $\mathcal B=\ell^\infty(\N_0)$, and let $Q_k$, $s_k$, $r_k$ be the maps of Theorem~\ref{thm:reservoir} with $a_n\equiv0$. Write $\epsilon_i(y)=\langle i|y|i\rangle$ for $i\in\{0,1\}$. Define
\[
 \bigl(\mathcal E_1(y\otimes b)\bigr)(j)=\epsilon_{j\bmod2}(y)\,b(s_1(j)),\qquad
 \mathcal E_k(y\otimes b)=\epsilon_0(y)\,Q_k^*b\quad(k\geq2).
\]
Every coordinate of $\mathcal E_1$ is the normal product state $\epsilon_{j\bmod 2}\otimes\mathrm{ev}_{s_1(j)}$. A unital positive map into a commutative von Neumann algebra whose coordinates are normal states is normal and completely positive. Thus $\mathcal E_1$ is admissible, and so is each $\mathcal E_k=\epsilon_0\otimes Q_k^*$. Moreover $\Phi_k=Q_k^*$ for all $k$.

For $x\in M_d$ at site $1$ and $\sigma=e_j$, Lemma~\ref{lem:local-restriction} gives
\[
 \varphi_N^{e_j}(x)=\bigl(Q_2\cdots Q_Ne_j\bigr)(y_x)=\epsilon_{t\bmod 2}(x),\qquad t=s_2\circ\cdots\circ s_N(j),
\]
because $y_x(t)=\epsilon_{t\bmod2}(x)$. Since $t\in\{u_1,v_1\}=\{4,5\}$ and $r_1(t)=t\bmod2$ for $t>1$, we get $t\bmod2=r_1(t)=r_N(j)$. With $a_n\equiv0$, $r_N(j)=0$ for $j\leq N$ and $r_N(j)=j\bmod 2$ for $j>N$. Hence, for every fixed probability vector $p$,
\[
 \varphi_N^{p}(x)\longrightarrow\epsilon_0(x).
\]
The relevant backward products are mixing: $Q_1\cdots Q_Np\to e_0$ by Theorem~\ref{thm:reservoir}, and $Q_2\cdots Q_Np\to e_4$, because $t=4$ whenever $r_N(j)=0$, in particular for all $j\leq N$. But for the boundary sequence $\sigma_N=e_{2N+3}$ one has $\varphi_N^{\sigma_N}(x)=\epsilon_1(x)$ for every $N$, so the local state at site $1$ depends on the boundary in the limit.

The same example lives on every infinite-dimensional $\sigma$-finite bond algebra. With $J,E$ as in Lemma~\ref{lem:retract}, replace $\mathcal E_k$ by $J\circ\mathcal E_k\circ(\id_{M_d}\otimes E)$. Since $EJ=\id$, the block maps become $J\circ\mathcal E_{[1,N]}\circ(\id\otimes E)$, and the finite-volume states become $\varphi_N^{J_*\sigma}$.
\end{remark}

\section{Consequences and limits of the extension}\label{sec:conclusion}

The preceding results give the following precise conclusions. First, uniform backward memory loss implies uniform convergence to a normal state on every von Neumann algebra, with no compactness hypothesis. Applying this from every starting time yields a unique normal entrance law and quantitative independence of terminal boundary conditions. This extends the backward contraction mechanism of the matrix setting at the uniform level.

Second, the corresponding pointwise assertion is universally valid exactly in finite dimension. In infinite dimension, nested, collectively compact state images and even finite-rank products do not suffice: the reservoir construction in Theorem~\ref{thm:reservoir} gives a counterexample on every infinite-dimensional $\sigma$-finite algebra. Thus the obstruction is not the distinction between types I, II, and III, nor merely the lack of compact output sets. It is the failure to control moving initial states when passing from pointwise to uniform convergence.

Third, the attained coefficient $\beta(T)$ makes the common-lower-bound criterion meaningful in the absence of rank-one projections, traces, or normal pure states. It provides an upper bound on the optimal contraction coefficient $\delta(T)$, not an equality in general and not a CP-order decomposition. The construction avoids a nonexistent noncommutative lattice infimum and retains normality by domination.

Fourth, the backward theory applies directly to quantum spin chains. For transfer maps with values in an arbitrary von Neumann algebra, uniform backward memory loss gives a unique thermodynamic limit, uniformly in the boundary condition. The same product coefficients bound the decay of correlations, which recovers the finite-dimensional matrix product state picture of \cite{Sou26,P26} for an infinite-dimensional bond algebra. Remark~\ref{rem:FCS-pointwise} shows that this uniformity cannot be relaxed to pointwise mixing.

Finally, mixing, ergodicity, and memory loss address three separate questions: whether a trajectory converges, whether its time average converges, and whether initial preparations become indistinguishable. Composition order supplies a fourth distinction, and uniformity over starting times supplies a fifth. The homogeneous benchmark removes some of these distinctions through the power structure; time-dependent channels do not have that mechanism. The examples above identify exactly which additional assumptions are needed in the stated theorems, rather than attributing convergence to an algebra type or to unproved compactness intuitions.

\end{document}